\documentclass[]{siamart220329}
\usepackage{bm}\usepackage{listings,amsmath,amssymb,graphicx,url,array,placeins,enumitem}
\usepackage{algorithm, setspace}
\usepackage{algpseudocode}
\usepackage{parskip}
\usepackage[toc,page]{appendix}
\usepackage[inner=3cm,outer=3cm,top=3cm,bottom=3cm]{geometry}
\usepackage{soul}
\usepackage[normalem]{ulem} 
\usepackage{caption} 
\usepackage{booktabs,dcolumn}
\newtheorem{problem}{Problem}

\newtheorem{example}{Example}
\let\oldexample\example
\renewcommand{\example}{\oldexample\normalfont}

\newcommand{\inv}{^{-1}}
\newcommand{\x}{x^k}
\newcommand{\xx}{x^{k+1}}

\newcommand{\M}{\mathcal{M}}
\newcommand{\N}{\mathcal{N}}

\newcolumntype{d}[1]{D{.}{.}{#1}}
\newcommand\mc[1]{\multicolumn{1}{c}{#1}}

\title{A Riemannian Gradient Descent Method for the Least Squares Inverse Eigenvalue Problem}
\headers{Least Squares Inverse Eigenvalue Problem}{Alban Bloor Riley, Marcus Webb, Michael L.~Baker}

\author{Alban Bloor Riley\footnote{Corresponding author.} \thanks{Department of Mathematics, The University of Manchester, Manchester M13 9PL, United Kingdom, \texttt{alban.bloorriley@manchester.ac.uk}, \texttt{marcus.webb@manchester.ac.uk}} \and Marcus Webb\footnotemark[2] \and Michael L.~Baker\thanks{The University of Manchester, Department of Chemistry, Manchester M13 9PL, United Kingdom, The University of Manchester at Harwell, Diamond Light Source, Harwell Campus, OX11 0DE, UK, \texttt{michael.baker@manchester.ac.uk}}}
\date{July 2024}

\begin{document}
\maketitle
\begin{abstract}
We address an algorithm for the least squares fitting of a subset of the eigenvalues of an unknown Hermitian matrix lying an an affine subspace, called the Lift and Projection (LP) method, due to Chen and Chu (SIAM Journal on Numerical Analysis, 33 (1996), pp.~2417--2430). The LP method iteratively `lifts' the current iterate onto the spectral constraint manifold then `projects' onto the solution's affine subspace. We prove that this is equivalent to a Riemannian Gradient Descent with respect to a natural Riemannian metric. This insight allows us to derive a more efficient implementation, analyse more precisely its global convergence properties, and naturally append additional constraints to the problem. We provide several numerical experiments to demonstrate the improvement in computation time, which can be more than an order of magnitude if the eigenvalue constraints are on the smallest eigenvalues, the largest eigenvalues, or the eigenvalues closest to a given number. These experiments include an inverse eigenvalue problem arising in Inelastic Neutron Scattering of Manganese-6, which requires the least squares fitting of 16 experimentally observed eigenvalues of a $32400\times32400$ sparse matrix from a 5-dimensional subspace of spin Hamiltonian matrices.
\end{abstract}

\section{Introduction} \label{sec. Introduction}
Let $A(x)$ be the affine family of matrices,
\begin{equation} \label{eqn: A(x)}
A( x) = A_0 + \sum^\ell_{i=1} x_i A_i,
\end{equation}
where $x\in\mathbb R^\ell$ and $A_0,\dots,A_\ell \in \mathbb R^{n\times n}$ are linearly independent symmetric matrices, and denote the ordered eigenvalues of $A(x)$ as $\lambda_1(x)\leq\dots\leq\lambda_n(x)$.
Then the least squares inverse eigenvalue problem (LSIEP) is given by:
\begin{problem}
[Least Squares Inverse Eigenvalue Problem \cite{chu_inverse_2005}]\label{prob: IEP}
Given real numbers $\lambda_1^*\leq\ldots\leq\lambda_m^*$, where $m \leq n$, find the parameters $x \in \mathbb R^\ell$ and the permutation $\rho$ $\in S_n$ that minimises
\begin{equation}\label{eqn: Objective function}
    F(x,\rho) = \frac 1 2 ||r(x,\rho)||^2_2 = \frac 1 2 \sum^m_{i=1}(\lambda_{\rho_i}( x) - \lambda_i^*)^2.
\end{equation} 
\end{problem}
The problem of finding a permutation $\rho$ can be considered subordinate to that of finding $x$, in the sense that for each parameter vector $x\in \mathbb{R}^\ell$, there is an essentially unique permutation that is optimal.
\begin{problem}[Eigenvalue Permutation Problem]\label{prob: rho}
Given real numbers $\lambda_1^*\leq\ldots\leq\lambda_m^*$, where $m \leq n$, and a parameter vector $x \in \mathbb R^\ell$, find a permutation $\rho \in S_n$ that minimises
\begin{equation} 
\sum^m_{i=1} \left(\lambda_{\rho_i}(x) - \lambda^*_i\right)^2.
\end{equation}
\end{problem}
This permutation minimises the difference between the prescribed eigenvalues, $\lambda_i^*$, and a subset of $m$ of the eigenvalues of $A(x)$. We will assume that such a permutation can be found, either by formulation as a linear sum assignment problem as in \cite[p.~203]{chu_inverse_2005}, or heuristically (see Section \ref{sec. rho}). Our focus for the remainder of the paper turns to the problem of iteratively finding the optimal $x$.

An important class of iterative methods is those of the form:   $\xx = \x+ p(\x)$ where each method is defined by the function $p$. A simple example is the gradient descent (steepest descent) algorithm, in which $p(x^k) = - \nabla F(x^k)$, where $\nabla F(x^k)$ is the gradient of $F$ at $x^k$. Given a fixed permutation $\rho \in S_n$, there is an explicit formula for the derivatives of the eigenvalues with respect to the parameters, commonly known as the Hellman--Feynman Theorem,
\begin{equation}
    \frac{\partial \lambda_{\rho_i}}{\partial x_j} = q_i(x)^T \frac{\partial A}{\partial x_j} q_i(x),
\end{equation}
where $q_i(x)$ is the eigenvector associated with the eigenvalue $\lambda_{\rho_i}(x)$ of $A(x)$. It can be proved simply by differentiating the relation $\lambda_{\rho_i} = q_i^T A q_i$. This gives us an explicit formula for $\nabla F$. Further, we can express it in the form $\nabla F = J_r(x)^Tr(x)$ where $J_r$ is the Jacobian matrix of $r(x)$, and
\begin{equation}\label{eqn. Jacobian}
    J_r(x) = \begin{pmatrix}
        q_1(x)^TA_1q_1(x)&\dots &q_1(x)^TA_\ell q_1(x)\\
        \vdots&\ddots&\vdots\\
        q_m(x)^TA_1q_m(x)&\dots& q_m(x)^TA_\ell q_m(x)
    \end{pmatrix}.
\end{equation}
  The cost of each of these iterations is relatively small as only the first derivative of $F$ is required, however the rate of convergence is merely linear. The Newton method \cite{deuflhard_newton_2011,nocedal_numerical_2006,dennis_numerical_1996}, where $p = -H_F\inv \nabla F$,  is another popular choice that offers locally quadratic convergence, but the cost of each iteration is higher as it requires  the calculation of the Hessian matrix, $H_F$. This method  may also require the modification of the Hessian away from a minimum if it is not positive definite. The Gauss-Newton method \cite{dennis_numerical_1996,nocedal_numerical_2006,bloor_riley_deflation_2025}, $p = -(J_r^TJ_r) \inv \nabla F$, is an alternative to the Newton method for least squares problems that does not require second derivatives but can enjoy locally quadratic convergence. Zhao, Jin and Yao have developed Riemannian Gauss--Newton and BFGS methods for the least squares inverse eigenvalue problem \cite{zhao_riemannian_2022, zhao_riemannian_2022-1}.

In this paper, we focus on the Lift and Projection (LP) method, due to Chen and Chu \cite{chen_least_1996, chu_inverse_2005}. This method, described in Section \ref{sec. LP}, iteratively `lifts' the current iterate onto the manifold of matrices with the constrained eigenvalues, then `projects' onto the affine subspace defined by $A(x)$. Given here as Algorithm \ref{alg. LP} we see that a full eigendecomposition is required to calculate $Z^k$, which makes this algorithm prohibitively slow for large matrices. Later in the section we prove Theorem \ref{Thrm. Alg equivalence}, which shows that Algorithm \ref{alg. LP} produces the same iterations as Algorithm \ref{alg. LP as GD}. In Section \ref{sec. Riemannian GD} we will introduce various tools from Riemannian geometry and ultimately prove the following theorem, which shows that the lift and projection method is actually a particular, and natural, example of a Riemannian gradient descent method (as described in Algorithm \ref{alg. RGD}):
\begin{theorem}\label{thrm. LP is RGD}
The lift and projection method is equivalent to the Riemannian gradient descent method on $\mathbb R^\ell $ equipped with the metric induced by $A(x)$ (in the sense of Definition \ref{def. Induced Metric}).
\end{theorem}

 The key advantage to this formulation is that only a partial eigendecomposition associated with the eigenvalues closest to $\lambda_1^*,\ldots,\lambda_m^*$ is required. Specialist eigensolvers such as FEAST \cite{polizzi_density-matrix-based_2009}, Jacobi-Davidson \cite{sleijpen_jacobi--davidson_2000} and Stewart's Krylov--Schur algorithm \cite{stewart_krylov--schur_2002} (used in MATLAB's \texttt{eigs} function) can be used to perform this. Even if, instead, a full eigendecomposition is computed, Algorithm \ref{alg. LP as GD} is cheaper to compute than Algorithm \ref{alg. LP} because the calculation of $Z^k$ is unnecessary.

Although these computational considerations motivated this work, it turns out that the reformulation of the LP method as an RGD method also allows more detailed analysis of the method. In Section \ref{sec. Global Convergence} we prove its global convergence properties, something already shown by Chen and Chu \cite{chen_least_1996}, with quantitative bounds on the decrease of the objective function.
\begin{theorem} \label{thrm. Global convergence of RGD LP}
    The Lift and Projection method satisfies
    \begin{equation}
        F(\xx) \leq F(\x) - \frac12 \nabla F(x^k)^T B^{-1} \nabla F(x^k),
    \end{equation}
    for all $k \geq 1$, where $F$ is the least squares objective as in Problem \ref{prob: IEP} and $B$ is the Gram matrix as in \eqref{eqn. Def B}. Hence, the method is strictly descending unless $\nabla F(x^k) = 0$.
\end{theorem}
Furthermore, we are able to generalise the method by augmenting the objective function, something less obvious in the original formulation of Lift and Projection. We present some numerical experiments in Section \ref{sec. Numerical Experiments}. 

\begin{algorithm}\caption{Lift and Projection method, as in \cite{chu_inverse_2005}, (LP)}\label{alg. LP}
\begin{algorithmic}
\State \textbf{given} $x^0 \in \mathbb{R}^\ell$, $\epsilon > 0$
\State Form the Gram matrix $B$ as in \eqref{eqn. Def B}
\For{$k = 0,1,2,\dots$}
\State Compute  the full eigendecomposition of $A(x^k)$
\State Calculate $\rho$ that solves Problem \ref{prob: rho}
\State Lift step: Form $Z^k$ by (\ref{OG Lift})
\State Project step: Calculate $x^{k+1}$ by (\ref{eqn: Old Project})
\If{$\|\xx-\x\| < \epsilon$} 
\State\Return $x^{k+1}$
\EndIf
\EndFor
\end{algorithmic}
\end{algorithm}

\begin{algorithm}\caption{Lift and Projection as Riemannian Gradient Descent, (RGD LP)}\label{alg. LP as GD}
\begin{algorithmic} 
\State \textbf{given} $x^0 \in \mathbb{R}^\ell$, $\epsilon > 0$
\State Form the Gram matrix $B$ as in \eqref{eqn. Def B}
\For{$k = 0, 1,2,\dots$}
\State Compute  the  partial eigendecomposition, $Q_1\Lambda_1 Q_1$,  of $A(x^k)$
\State Calculate the gradient $\nabla F(\x) = J_r(\x)^Tr(\x)$ by \eqref{eqn. Jacobian}  
\State Form the step $p^k = - B^{-1}\nabla F(\x)$
\State $x^{k+1} = x^k +p^k$
\If{$\left((p^k)^T B p^k\right)^{1/2} < \epsilon$} 
\State\Return $x^{k+1}$
\EndIf
\EndFor
\end{algorithmic}
\end{algorithm}

\section{Lift and Projection}\label{sec. LP}
The Lift and Projection method due to Chen and Chu \cite{chen_least_1996},  is a method specifically designed for solving the least squares inverse eigenvalue problem. It consists of two operations repeated iteratively: the lift step of mapping to the nearest point on the manifold of matrices with the correct eigenvalues, and the projection step of orthogonally projecting onto the space of matrices with the required structure.
\subsection{Lift}
The lift step, at each iteration $k$, finds the matrix with the desired spectrum that is nearest to  $A(x^k)$ in the Frobenius norm. Since $A(x^k)$ is a linear combination of Hermitian matrices it is also Hermitian and therefore we can write its spectral decomposition as
\begin{equation}
    A(x^k) = Q(x^k)\Lambda(x^k)Q(x^k)^T.
\end{equation}
Note that it is possible to reorder the eigenvalues and corresponding eigenvectors such that the decomposition becomes
\begin{equation}\label{eqn: eigendecomposition}
    A( x^k) =     Q\begin{pmatrix}
    \Lambda_{1}(x^k) &0\\
    0& \Lambda_{2}( x^k)
    \end{pmatrix}Q^T = \begin{pmatrix}
        Q_1 & Q_2
    \end{pmatrix}
    \begin{pmatrix}
    \Lambda_{1}(x^k) &0\\
    0& \Lambda_{2}(x^k)
    \end{pmatrix}\begin{pmatrix}
        Q_1^* \\ Q_2^*
    \end{pmatrix}
\end{equation}
where $\Lambda_1( x^k) = \operatorname{diag}(\lambda_{\rho_1},\dots,\lambda_{\rho_m})$ is the diagonal matrix of eigenvalues of $A(x)$ closest to the prescribed eigenvalues; $\Lambda_{2}( x^k) = \operatorname{diag}(\lambda_{\rho_{m+1}},\dots,\lambda_{\rho_n}) $ is the diagonal matrix formed of the remaining eigenvalues of $A(x)$ and $Q_1, Q_2$ are the matrices of corresponding eigenvectors.

Using \cite{brockett_dynamical_1991}, we see that the  matrix, $Z^k$, with the desired eigenvalues that is closest to $A(x^k)$ is given by simply replacing  $\Lambda_1$ with the matrix of prescribed eigenvalues  $\Lambda^* = \operatorname{diag}(\lambda_1^*,\dots\lambda_m^*)\in \mathbb R^{m\times m}$,
\begin{equation}\label{OG Lift}
    Z^{k} = Q(x^k)
    \begin{pmatrix}
    \Lambda^* &0\\
    0& \Lambda_{2}(x^k)
    \end{pmatrix}Q(x^k)^T.
\end{equation}
The cost of this lift step is one full eigendecomposition, one  solve of Problem \ref{prob: rho} to find the permutation $\rho$, and two  matrix-matrix multiplications, which is $O(n^3)$ in total. 

\subsection{Projection}
The projection step  finds the next iterate, $x^{k+1}$, such that  $A(x^{k+1})$ is the matrix that minimises
\begin{equation}
    \|Z^k - A(x^{k+1})\|_F.
\end{equation}
This step can be calculated by solving the linear system 
\begin{equation}\label{eqn: Old Project}
    Bx^{k+1} = c,
\end{equation}
where $B \in \mathbb R^{\ell\times \ell}$ is  the Gram matrix
\begin{equation}\label{eqn. Def B}
    B =  
    \begin{pmatrix}
        \langle A_1, A_1 \rangle_F & \langle A_1, A_2 \rangle_F & \cdots \\
        \langle A_2, A_1 \rangle_F & \ddots & \vdots \\
        \vdots & \cdots & \langle A_\ell, A_\ell \rangle_F
    \end{pmatrix},
\end{equation}
given by the  Frobenius inner product, and $c(x^k)\in \mathbb R^\ell$ is the vector
\begin{equation}\label{eqn. Def c}
    c = 
    \begin{pmatrix}
         \langle Z^k-A_0, A_1\rangle_F\\
         \vdots\\
          \langle Z^k-A_0, A_\ell\rangle_F
    \end{pmatrix}.
\end{equation}
Notice that $B$ does not depend on $x^k$, and so does not need to be calculated each iteration. The cost of this step is two triangular solves, given a precomputed Cholesky factorisation of $B$, which is negligible compared to the eigendecomposition.

Chen and Chu  showed that the Lift and Projection method is a globally convergent algorithm, with the following result.
\begin{theorem}[Theorem 4.1 from \cite{chen_least_1996}]\label{thrm. Chen Chu descent algorithm}
The Lift and Projection algorithm is a descent method in the sense that
\begin{equation}\label{eq. Chu Chen descent thrm}
    || A(x^{k+1})-Z^{k+1}||_F  \leq || A(x^{k+1}) - Z^{k}||_F \leq || A(x^{k}) - Z^{k}||_F,
\end{equation}
for $k = 0,1,2,\ldots$.
\end{theorem}
In Section \ref{sec. Global Convergence} we will prove a stronger version of this theorem.
\subsection{Algorithm Equivalence}
We shall now prove the following Theorem.
\begin{theorem}\label{Thrm. Alg equivalence}
Algorithm \ref{alg. LP as GD} and  Algorithm \ref{alg. LP} are equivalent algorithms in the sense that they produce exactly the same sequence of  iterates, $x^k$.
\end{theorem}
First we prove a Lemma that describes the lift step of the algorithm.
\begin{lemma} \label{Lem: Form Z}
The Lift step can be written as
\begin{equation}\label{eqn: Lift step}
\quad  Z^k  =A(x^k) -  Q_1 \, \Delta \Lambda \,Q_1^T
\end{equation}
where $Q_1\in \mathbb R^{n\times m}$  is as in \eqref{eqn: eigendecomposition}, and $\, \Delta \Lambda \, = \operatorname{diag}(r(x,\rho))$.
\end{lemma}

\begin{proof}
First we recall the formula for $A(x)$ given by  \eqref{eqn: eigendecomposition}, and $Z^k$ given by \eqref{eqn: Lift step}. We will then show that $Q_2$ is in fact not needed in the calculation of $Z^k$. Looking  at the difference between $A(x^k)$ and $Z^k$ we see that 
\begin{align}
    Z^{k} - A(x^k) &= 
    \begin{pmatrix}
Q_1& Q_2
\end{pmatrix}
    \begin{pmatrix}
    \Lambda^* &0\\
    0& \Lambda_{2}
    \end{pmatrix}\begin{pmatrix}
Q_1^T\\ 
Q_2^T
\end{pmatrix}
  -
  \begin{pmatrix}
Q_1& Q_2
\end{pmatrix}
    \begin{pmatrix}
    \Lambda_{1} &0\\
    0& \Lambda_{2}
    \end{pmatrix}\begin{pmatrix}
Q_1^T\\ 
Q_2^T
\end{pmatrix}
\\
     &= \begin{pmatrix}
Q_1& Q_2
\end{pmatrix}
    \begin{pmatrix}
    \Lambda^* -\Lambda_{1} &0\\
    0& 0
    \end{pmatrix}\begin{pmatrix}
Q_1^T\\ 
Q_2^T
\end{pmatrix} \\
&= Q_1(\Lambda^* -\Lambda_{1})Q_1^T\\
&= - Q_1 \, \Delta \Lambda \, Q_1^T. 
\end{align}
\end{proof}

We will now prove Theorem \ref{Thrm. Alg equivalence}.
\newcommand{\1}{_1}
\begin{proof}[Proof of Theorem \ref{Thrm. Alg equivalence}]
We recall from \eqref{eqn. Def c} that
\begin{equation}
    c_j = \langle Z^k-A_0,A_j\rangle_F,
\end{equation} 
which we can rewrite using Lemma \ref{Lem: Form Z} as
\begin{equation}
     c_j = \langle A(x^k) -Q\1 \,\Delta \Lambda\,  Q^T\1-A_0,A_j\rangle_F.
\end{equation}

Then, since the Frobenius inner product is bilinear we can separate the equation for $c_j$ into
\begin{equation}\label{eqn: c components}
     c_j = 
     \langle A(x^k) -A_0,A_j\rangle_F -\langle Q\1 \, \Delta \Lambda \, Q\1^T,A_j\rangle_F .
\end{equation}
We now look at both of these terms individually. Using bilinearity again we can expand the first term as
\begin{equation}
    \langle A(x^k)-A_0,A_j\rangle_F = \sum^\ell_{i=1} \langle A_i,A_j\rangle_F x_i^k = \sum^\ell_{i=1} B_{i,j}x_i^k = b_j^Tx^k
\end{equation}
where $B$ is the Gram matrix $B$ defined in \eqref{eqn. Def B} and $b_j$ is its $j$th column. By using the definition of Frobenius inner product and the symmetry of $\, \Delta \Lambda \,$ we can expand the second term of \eqref{eqn: c components}  as
\begin{align}
    \langle Q\1 \, \Delta \Lambda \, Q\1^T,A_j\rangle_F &= \operatorname{Tr}((Q\1 \, \Delta \Lambda \, Q\1^T)^TA_j) = \operatorname{Tr}(Q\1 \, \Delta \Lambda \, Q\1^TA_j) = 
    \operatorname{Tr}(\, \Delta \Lambda \, Q\1^TA_jQ\1)\\
    &=\sum^m_{i=1} \, \Delta \Lambda \,_{i,i} (Q\1^TA_jQ\1)_{i,i} = \sum^m_{i=1} r_i q_i^TA_jq_i,
\end{align}
where $r_i$ is the $i$th component of the residual $r(x,\rho)$, and $q_i$ is the $i$th column of $Q\1$. Note that by \eqref{eqn. Jacobian} we can  write this in terms of the Jacobian matrix $J_r(x) \in \mathbb R^{m\times l}$ as
\begin{align}
   \sum^m_{i=1} r_i q_i^TA_jq_i = \sum^m_{i=1} r_i (J_r)_{i,j} = \sum^m_{i=1} (J_r^T)_{j,i}r_i  = (J_r^Tr)_{j}.
\end{align}
Thus  \eqref{eqn: c components} becomes
\begin{equation}
    c =Bx^k -  J_r^Tr.
\end{equation}
Then, substituting this into \eqref{eqn: Old Project} we get
\begin{equation}
      Bx^{k+1} = Bx^k -  J_r^Tr.
\end{equation}
Multiplying by $B^{-1}$ and using the fact that $\nabla F = \frac 1 2 \nabla  r^Tr =  J_r^Tr$ we get
\begin{equation}\label{LP as Newton step}
    x^{k+1} = x^k - B^{-1}\nabla F.
\end{equation}
We know that $B$ is invertible since it is assumed that the $A_i$ are linearly independent.

Thus, we have shown that the sequence of steps taken by the Lift and Projection algorithm given by Algorithm \ref{alg. LP} is equivalent to the steps, $\xx = \x - B\inv\nabla F$, in Algorithm \ref{alg. LP as GD}.
\end{proof}  

\section{Riemannian Gradient Descent}\label{sec. Riemannian GD}

Riemannian Gradient Descent is a generalisation of the standard gradient descent method to a Riemannian manifold, not just standard Euclidean space. We will then prove Theorem \ref{thrm. LP is RGD}, that is, the equivalence of Lift and Projection method and a certain Riemannian Gradient Descent method.


\begin{algorithm}\caption{Riemannian Gradient Descent, as in \cite{boumal_introduction_2023}}\label{alg. RGD}
\begin{algorithmic}
\State Let $\M$ be a manifold, equipped with Riemannian metric $g$
\State Initial guess = $x^0 \in \M$ 
\For{$k = 1,2,\dots$}
\State Calculate $\nabla_g F(x^k)$, the gradient of $F(x)$ on $\M$
\State $\xx = \x - \nabla_g F(x^k)$
\EndFor
\end{algorithmic}
\end{algorithm}

What makes a `Riemannian' gradient descent method different is that the Riemannian gradient, $\nabla_g F$, is calculated with respect to the Riemannian metric associated with the manifold $\M$. When $\M = \mathbb R^n$ and $g$ is the standard Euclidean metric, i.e.~the dot product,  then this method simplifies to the standard gradient descent method.

\subsection{Riemannian Geometry}
We will now introduce some definitions and concepts from Riemannian geometry to discuss this method. 
\begin{definition}[{\cite[p.~194]{boumal_introduction_2023}}]
Let $\mathcal{M}$ be a smooth manifold. An inner product on $T_x\M$ is a bilinear, symmetric, positive definite function $g_x : T_x\M \times T_x\M \rightarrow\mathbb R$. It induces a norm for tangent vectors: $\|u\|_{g_x} =  g_x(u,u)$. A metric, $g$,  on $\M$ is a choice of inner product $g_x$ for each $x \in \M$.
\end{definition}
\begin{definition}[{\cite[p.~194]{boumal_introduction_2023}}]
    A metric $g$ on $\M$ is a Riemannian metric if it varies smoothly with $x$. A Riemannian manifold is a manifold with an associated Riemannian metric, which can be denoted as the pair $(\M, g)$.
\end{definition}

\begin{definition}[{\cite[p.~18]{boumal_introduction_2023}}]
Let $\varphi : \M \rightarrow \N$ be smooth function between Riemannian manifolds $\M$ and $\N$, and let $\gamma:\mathbb [-1,1]\rightarrow\M$ be any differentiable curve such that $\gamma(0) = x, \gamma'(0)=v $. 
Then the directional derivative of $\varphi$ at $x \in \mathcal{M}$ in the direction $v \in T_x\mathcal{M}$ is given by
\begin{equation}
     D\varphi (x)[v] = \left.\frac{d}{dt} \varphi\circ\gamma(t)\right|_{t=0} =  \lim_{t\rightarrow0} \frac{\varphi (\gamma(t)) - \varphi (x)}{t}.
\end{equation}
This defines the differential, $D\varphi(x) : T_x\M \to T_{\varphi(x)}\N$, of $\varphi$ at $x$.
\end{definition}

\begin{definition}[{\cite[p.19]{boumal_introduction_2023}}]\label{def. gradient}
Let $F : \M \rightarrow \mathbb R$ be smooth on a Riemannian manifold $\M$. The Riemannian gradient of $F$ on $\M$ is the vector field  $\nabla_g F : \mathcal{M} \to T\mathcal{M}$  defined by the identity
\begin{equation}
    DF(x)[v] = g( v, \nabla_g F(x)), \qquad \forall (x,v) \in T\M .
\end{equation}
\end{definition}
The gradient $\nabla_g F(x)$ is dependent on the choice of metric $g$, whereas the differential of $F$ is not. In some cases there is  a `natural' choice of metric. Often the Riemannian manifold we would like to optimise on, $\M$  is a submanifold of some larger manifold with a defined metric, this allows us to induce a metric on $\M$.


\begin{definition}
    Let $A:\M\rightarrow\N$ then $A$ is an immersion if its differential $DA(x)$ is injective at every $x\in \M$.
\end{definition}

\begin{definition}[\cite{lee_introduction_2018}]\label{def. Induced Metric}
Let $\N$ be a Riemannian manifold equipped with metric $ g$, and let  $\M$ be a smooth manifold. Suppose $A:\M\rightarrow\N$ is an immersion, then the induced metric $g^A$ on $\M$ (induced by $A$) is given by
\begin{equation}
    g^A(u,v) :=  g(DA(x)[u],DA(x)[v]),
\end{equation}
for all $u,v \in T\M_x$
\end{definition}
In fact by the Nash Embedding Theorem every manifold is a submanifold of Euclidean space, which allows us to use the Euclidean metric restricted to $T\M$.

\subsection{The Lift and Projection method as a Riemannian gradient descent method}
In this Section we prove Theorem \ref{thrm. LP is RGD}.

\begin{lemma}\label{Lemma. induced gradient}
    Let $\M$ and $\N$ be manifolds where $\N$ is equipped with the metric $g$. Let  $A:\M\rightarrow\N$ be an immersion and  $F:\M\rightarrow \mathbb R$ be a differentiable function. Then $\nabla_{g^A}F$, that is the gradient of $F$ on the manifold $\M$ with metric $g^A$ induced by $A$, is defined by
    \begin{equation}
         DF(x)[v] = g( DA(x) v, DA(x) \nabla_{g^A}F(x)), \qquad \forall (x,v) \in T\M.
    \end{equation}
    Therefore,
    \begin{equation}
        \nabla_{g^A} F(x) =  (DA(x)^* DA(x))\inv \, \nabla_g F(x),
    \end{equation}
    where $DA(x)^*$ is the adjoint operator of $DA(x) : T_x\M \to T_{A(x)} \N$ taken with respect to the metric $g$.
\end{lemma}
\begin{proof}
    The first part of the lemma  follows from Definitions \ref{def. Induced Metric} and \ref{def. gradient}. 
    Then, using the definition of adjoint, we have for any $(x,v) \in T\mathcal{M}$
    \begin{align}
        DF(x)[v] &= g( DA(x) v, DA(x) \nabla_{g^A}F(x)) \\
        &= g(v, DA(x)^* DA(x) \nabla_{g^A} F(x)).
    \end{align}
    By Definition \ref{def. gradient}, we therefore have
    \begin{equation}
        \nabla_g F(x) = DA(x)^* DA(x) \nabla_{g^A} F(x).
    \end{equation}
Since $DA(x)$ is injective for all $x$, we have that $DA(x)^*DA(x)$ is a positive-definite linear operator on $T_x\mathcal{M}$, so we can invert it to obtain the second part of the lemma.
\end{proof}
We will now prove Theorem \ref{thrm. LP is RGD}.
\begin{proof}[Proof of Theorem \ref{thrm. LP is RGD}]
Let $F(x):\mathbb R^\ell \rightarrow\mathbb R$ be as in \eqref{eqn: Objective function}, $A(x): \mathbb R^\ell \rightarrow\mathbb R^{n\times n}$ be as in \eqref{eqn: A(x)}, and $g$ be the Euclidean (Frobenius) metric, then by Definition \ref{def. Induced Metric} 
 \begin{align}
 g^A(u,v) &= \langle DA(x)[u], DA(x)[v]\rangle_F= \langle DA[u], DA[v]\rangle_F  \\
 &= \langle \sum^\ell_{i=1} A_iu_i,  \sum^\ell_{i=1} A_iv_i\rangle_F =  \sum^\ell_{i=1} \sum^\ell_{j=1} \langle A_iu_i, A_jv_j\rangle_F\\
 &= \sum^\ell_{i=1} \sum^\ell_{j=1} u_i v_j \langle A_i, A_j\rangle_F\\
 &= u^TBv
 \end{align}
Therefore  $DA(x)^*DA(x) = B$, for $B$  as defined in \eqref{eqn. Def B}.
Therefore by Lemma \ref{Lemma. induced gradient} with $\M = \mathbb R^\ell$, $\N=\mathbb R^{n\times n}$ and $g$ the Euclidean (Frobenius) metric on $\mathbb R^{n\times n}$, the gradient of $F$ on the manifold $\mathbb R^\ell$ with metric induced by $A$ is 
\begin{align}
     \nabla_{g^A} F(x)  = (DA(x)^* DA(x))\inv \nabla_gF = B\inv \nabla F.
\end{align}
Therefore the Riemannian gradient descent method  (Algorithm \ref{alg. RGD}) using gradient $\nabla_{g^A} F(x)$ is equivalent to Algorithm \ref{alg. LP as GD} which by Theorem \ref{Thrm. Alg equivalence} is also equivalent to Algorithm \ref{alg. LP}.
\end{proof}



\section{Global Convergence}\label{sec. Global Convergence}
Chen and Chu showed, in Theorem \ref{thrm. Chen Chu descent algorithm}, that each step of  the Lift and Projection algorithm will not result in an increase in the objective function. This however does not guarantee  convergence  since none of the inequalities in \eqref{eq. Chu Chen descent thrm} are strict. In this Section we  will prove a  stronger result showing a decrease at all iterates before convergence is achieved. First we  prove a new result about the relationship between the Hessian of $F$, $H_F$, and the metric tensor $B$ as defined in \eqref{eqn. Def B}. Note that, by \cite{chen_least_1996,chu_inverse_2005}, we can write this Hessian as
\begin{equation}\label{eqn:Hessian}
H_F = J_r^TJ_r + \sum^m_{k=1} r_k H_{r_k},
    \qquad\text{where}\qquad(H_{r_k})_{ij} = 2\sum^m_{\substack{t=1\\\lambda_t\neq\lambda_k}} \frac{(q_t^TA_iq_k)(q_t^TA_jq_k)}{\lambda_k-\lambda_t}.
\end{equation}
We have written $\lambda_k$ instead of $\lambda_{\rho_k}$ for brevity. Recall from the introduction that $q_k$ is the eigenvector associated with the eigenvalue $\lambda_k$. The following is a key relationship between this Hessian and the Gram matrix, $B$, from equation \eqref{eqn. Def B}.
\begin{lemma}\label{Lem. B-H_F posdef}
    $B\geq H_F$ with respect to the Loewner order, that is, $B-H_F$ is positive semidefinite. 
\end{lemma}
\begin{proof}
Define $\tilde{A}(x) = A(x) - A_0$. To prove that $B-H_F$ is positive semi definite we will show that $v^TBv\geq v^TH_Fv$ for all $v\in\mathbb R^\ell$. First we expand $v^TBv$ as
\begin{align}
    v^TBv   &= \sum^\ell_{i=1} \sum^\ell_{j=1} v_iv_jB_{ij} = \operatorname{Tr}(\sum^\ell_{i=1} \sum^\ell_{j=1} v_iv_jA_iA_j) =\operatorname{Tr}(\sum^\ell_{i=1} v_i A_i \sum^\ell_{j=1} v_jA_j) = \left\| \tilde{A}(v) \right\|_F^2.
    \end{align}
Note that $\|Q_1^T \tilde{A}(v) Q_1 \|_F \leq \|Q_1\|_2^2\|\tilde{A}(v)\|_F = \|\tilde{A}(v)\|_F$ by \cite[problem 6.5]{higham_accuracy_2002}, so
    \begin{align}
    v^T B v &\geq \|Q_1^T \tilde{A}(v) Q_1 \|_F \label{eq. QAQQAQ} \\
    &= \sum^m_{t=1}\sum^m_{k=1} (q_k^T\tilde{A}(v)q_t)^2 \\
    &= \sum^m_{k=1}(q_k^T\tilde{A}(v)q_k)^2 + 2\sum^m_{k=1}\sum^{k-1}_{t=1}(q_t^T\tilde{A}(v)q_k)^2.\label{eq. v^TBv}
\end{align}
Now we expand $v^TH_Fv$. We use equation \eqref{eqn:Hessian} and define $S = \sum^m_{k=1} r_k H_{r_k}$ for brevity, and obtain,
\begin{align}
    v^TH_Fv & = v^T(J_r^TJ_r+S)v= v^T(J_r^TJ_r)v+ v^TSv  \\
    &= v^T(J_r^TJ_r)v +2\sum^m_{k=1} \sum^m_{\substack{t=1\\\lambda_t\neq\lambda_k}}\left( \frac{\lambda_k-\lambda_k^*}{\lambda_k-\lambda_t}\sum^\ell_{i=1} \sum^\ell_{j=1}v_i \, q_t^TA_iq_k \, v_j \, q_t^TA_jq_k\right)\\
    &= v^T(J_r^TJ_r)v +2\sum^m_{k=1} \sum^m_{\substack{t=1\\\lambda_t\neq\lambda_k}} \frac{\lambda_k-\lambda_k^*}{\lambda_k-\lambda_t} \, q_t^T\tilde{A}(v)q_k \, q_t^T\tilde{A}(v)q_k\\
    &= (J_rv)^TJ_rv +2\sum^m_{k=1} \sum^{k-1}_{\substack{t=1\\\lambda_t<\lambda_k}} \frac{\lambda_k-\lambda_k^*}{\lambda_k-\lambda_t}(q_t^T\tilde{A}(v)q_k)^2+ \frac{\lambda_t-\lambda_t^*}{\lambda_t-\lambda_k}(q_k^T\tilde{A}(v)q_t)^2\\
    &= \|J_rv\|_2^2 + 2\sum^m_{k=1} \sum^{k-1}_{\substack{t=1\\\lambda_t<\lambda_k}} \left(\frac{\lambda_k-\lambda_k^*}{\lambda_k-\lambda_t}+ \frac{\lambda_t^*-\lambda_t}{\lambda_k-\lambda_t}\right)(q_t^T\tilde{A}(v)q_k)^2\\
    & = \sum^m_{k=1}((J_rv)_k)^2 + 2\sum^m_{k=1} \sum^{k-1}_{\substack{t=1\\\lambda_t<\lambda_k}} \left(1- \frac{\lambda_k^*-\lambda_t^*}{\lambda_k-\lambda_t}\right)(q_t^T\tilde{A}(v)q_k)^2\\
    &\leq \sum^m_{k=1} (q_k^T\tilde{A}(v)q_k)^2 + 2\sum^m_{k=1} \sum^{k-1}_{\substack{t=1\\\lambda_t<\lambda_k}}(q_t^T\tilde{A}(v)q_k)^2 \label{eq. v^TH_Fv}.
\end{align}
The final inequality holds because the prescribed eigenvalues and eigenvalues of the current iterate are sorted, so $(\lambda_k^* - \lambda_t^*)/(\lambda_k - \lambda_t) > 0$. Therefore, by \eqref{eq. v^TBv} and \eqref{eq. v^TH_Fv},
\begin{align*}
    v^T(B-H_F)v \geq 2\sum_{k=1}^m \sum^{k-1}_{\substack{t=1\\\lambda_t=\lambda_k}} (q_t^T\tilde{A}(v)q_k)^2 \geq 0.
\end{align*}
Since $v\in\mathbb{R}^\ell$ is arbitrary, $B \geq H_F$.
 \end{proof}
Using this lemma we can now prove Theorem \ref{thrm. Global convergence of RGD LP},  a stronger result about the convergence of the Lift and Projection method than that proved in \cite{chen_least_1996}.

\begin{proof}[Proof of Theorem \ref{thrm. Global convergence of RGD LP}]
Recall that we wish to show $F(\xx) \leq F(\x) - \frac12 \nabla F(x^k)^T B^{-1} \nabla F(x^k)$. The Taylor expansion of $\nabla F$ around $x$ is%
\begin{equation}
    F(y)-F(x) =  \nabla F(x)^T (y-x) +\frac 1 2 (y-x)^T H_F(\xi)(y-x), 
\end{equation}
where $\xi = x + t(y-x)$ for some $t \in (0,1)$. Then if we substitute $x = \x, y = \xx$ and let $p^k =\xx-\x =  -B\inv\nabla F$ then
\begin{align}
    F(\xx) - F(\x) &= \nabla F(x)^T (\xx-\x) +\frac 1 2 (\xx-\x)^TH_F(\xx-\x) \\
    & = \nabla F(\x)^Tp^k+ \frac 1 2 (p^k)^TH_Fp^k\\
    & = \nabla F(\x)^TB\inv Bp^k+ \frac 1 2 (p^k)^TH_Fp^k\label{eq. sub B}\\
    & = -(p^k)^TBp^k+ \frac 1 2 (p^k)^TH_Fp^k\\
    & \leq -(p^k)^TBp^k+ \frac 1 2 (p^k)^TBp^k \\
    &= - \frac 1 2 (p^k)^TBp^k\\
    &= - \frac12 \nabla F(x^k)^T B^{-1} \nabla F(x^k).
\end{align}
 The inequality comes from $v^T H_F v \leq v^T B v$, which follows from Lemma \ref{Lem. B-H_F posdef}.
\end{proof}
Thus the method is globally convergent -- that is the method produces a decrease in the value of the objective function, unless the iterates have converged to a stationary point of $F$. It is possible to  extend this proof further to steps of double the length.
\begin{corollary}
    The `doubled' Lift and Projection method, that is the method defined by the step $\xx = \x - 2B\inv\nabla F$, is also a descent method, in the sense that 
    \begin{equation}
        F(\xx) \leq F(\x).
    \end{equation}
\end{corollary}
\begin{proof}
    This follows from the above proof, with $p = -2B\inv\nabla F$:
\begin{align}
    F(\xx) - F(\x)  & = \frac 1 2 \nabla F(\x)^T2B\inv Bp+ \frac 1 2 p^TH_Fp\\
    & = -\frac 1 2p^TBp+ \frac 1 2 p^TH_Fp,\\
    & = \frac 1 2 p^T( H_F-B)p \leq 0   
\end{align}
again because $B-H_F$ is positive semidefinite, by Lemma \ref{Lem. B-H_F posdef}.
\end{proof}

\section{Numerical Experiments}\label{sec. Numerical Experiments} \label{sec. rho}

In this paper we have discussed solving the least squares inverse eigenvalue problem, where there are $m$ prescribed eigenvalues of an element of an $\ell$-dimensional space of $n\times n$ symmetric matrices. All of these examples can be reproduced using code available in the GitHub repository \url{https://github.com/AlbanBloorRiley/RGD_LP}.

In Section \ref{sec. LP} we showed that Algorithms \ref{alg. LP} and \ref{alg. LP as GD} produce the same sequence of steps, but this does not mean that the algorithms are necessarily identical in practice. The key advantage of Algorithm \ref{alg. LP as GD} is that only a partial eigendecomposition associated with the solution to Problem \ref{prob: rho} (loosely speaking, the closest eigenvalues to $\lambda_1^*,\ldots,\lambda_m^*$). However, it is not so simple in general. Specialist eigensolvers such as FEAST \cite{polizzi_density-matrix-based_2009}, Jacobi-Davidson \cite{sleijpen_jacobi--davidson_2000} and Stewart's Krylov--Schur algorithm \cite{stewart_krylov--schur_2002} (used in MATLAB's \texttt{eigs} function) can be used to compute a subset of the eigenvalues based on a criterion, such as: eigenvalues within an interval, eigenvalues with smallest or largest (in value or magnitude), or eigenvalues closest to a given number. If more information is known about the presribed eigenvalues, such as that they are the smallest $m$ eigenvalues of the matrix, then \texttt{eigs} can be used reliably for this partial eigendecomposition. Otherwise, we may still need to compute a full eigendecomposition if we insist on finding the closest eigenvalues to $\lambda_1^*,\ldots, \lambda_m^*$. In this worst case scenario, Algorithm \ref{alg. LP as GD} is still cheaper to compute than Algorithm \ref{alg. LP} because the calculation of $Z^k$ is unnecessary. 

In Examples 2, 3 and 4, we consider the algorithm we denote by \emph{RGD LP Min} that performs Algorithm \ref{alg. LP as GD} in which the partial eigendecomposition step computes the $m$ eigenvalues that are smallest (in value), rather than those that solve Problem \ref{prob: rho}. This is equivalent to setting $\rho_i = i$ and is much faster to compute. This scenario is not uncommon in applications.

Another advantage of using algorithms such as FEAST, Jacobi--Davidson and Stewart's Krylov--Schur algorithm is that sparsity can be taken advantage of, whereas classical full eigensolvers cannot avoid fill-in.

 \begin{example}\label{Ex. toeplitz}

The first example we will look at is an underdetermined Symmetric Toeplitz inverse eigenvalue problem. Define the basis matrices as $A_0 = 0$ and 
\begin{equation}
    {A_k}= \sum_{|i-j| = k} e_ie_j^T \in \mathbb R^{5000\times5000} . 
\end{equation}
This example will be an underdetermined system with $\ell = 40 $ basis matrices and $m = 20$ prescribed eigenvalues $\Lambda^* = (-110, -109.8, -109.6, \dots, -106.2)^T$ and the initial guess be $x_0 = (1,\dots,1)^T$. The  stopping criterium for all methods is set to  $\|\xx-\x\| < 0.0001$. The computation cost is summarised in Table \ref{tab: toeplitz}. As can be seen in the table all methods require the same number of iterations. Clearly both of the Riemannian Gradient Descent methods are significantly faster compared to the original method, with the fastest speedup achieved by the RGD LP Min algorithm. 

    \begin{table}
     \centering
     \caption{Computational cost for Example \ref{Ex. toeplitz}, with stopping tolerance $\epsilon = 0.0001$}
     \label{tab: toeplitz}
     \begin{tabular}{c*{4}{d{3.3}}}
     \toprule
           \mc{Algorithm} & \mc{No. Iterations} & \mc{CPU Time (seconds)} & \mc{Time per Iteration}\\\hline
          RGD LP Min &   20   &  6.4 & 0.32 \\
          RGD LP          &   20   &  185.2 & 9.26\\
          LP              &   20   &  475.82 & 23.79 \\\toprule
          \end{tabular}
     
 \end{table}

 \end{example} 
 \begin{example}\label{Ex. Mn12}

The rest of examples arise from inelastic neutron scattering (INS) experiments. INS is a spectroscopic technique used to measure the magnetic excitations in materials with interacting electron spins, such as single ions or molecular-based magnet. The  experiments are able to measure the energy between quantum spin states, these energy are then associated with the eigenvalues of the Hamiltonian matrix that describes the quantum spin dynamics of the compound in question \cite{furrer_magnetic_2013,baker_neutron_2012, baker_spectroscopy_2014}. This information is crucial in understanding quantum phenomena and potentially can help  utilise electronic quantum spins in new quantum applications such as information sensing and processing. The particular parametrised inverse eigenvalue problem that arises from an INS experiment is determined by the Hamiltonian model that is used to describe the spin system, and the eigenvalues that have been found experimentally. It is important to note here that the experiment does not actually measure the eigenvalues directly -- it is the differences between the eigenvalues. Thus we add an extra parameter to the system, $A_{\ell+1} = I$, to find the `ground state' of the system, that is the value of the smallest eigenvalue. As a consequence of these experiments being done at very cold temperatures it is always the lowest energy transitions that are found, which thus correspond to the smallest eigenvalues. Note that many of the  the basis matrices used to model the Hamiltonian are highly sparse, as shown in Example \ref{Ex. Cr6} Figure \ref{fig: SpyPlots}. Further numerical examples solving inverse eigenvalue problems arising in INS by a deflated Gauss--Newton method can be found in \cite{bloor_riley_deflation_2025}.

The first INS example we will look at is that of Manganese-12-acetate. This particular molecule became important when it was discovered that it acts like a nano-sized magnet with a molecular magnetic coercivity and the identification of quantum tunnelling of magnetisation see \cite{friedman_macroscopic_1996, sessoli_magnetic_1993}. The spin Hamiltonian matrix of this system is a $21\times21$ matrix  and and can be modelled with 4 basis matrices \cite{bircher_transverse_2004}
\begin{equation}\label{Ham}
    A_1 = O^0_2,\, A_2 = O^0_4,\, A_3 = O^2_2,\, A_4 = O^4_4 \, \in \mathbb R ^{21\times 21}
\end{equation}
where  $O_k^q$ are Stevens operators, see \cite{gatteschi_molecular_2006} and \cite{rudowicz_generalization_2004} for more general details on Stevens operators. Recall that the ground state basis matrix (as well as the ground state matrix $A_{5} = I$) is also required. These operators, calculated using the EasySpin MATLAB package \cite{stoll_easyspin_2006},  are defined in this case with a spin of $S = 10$, as
\begin{align*}
O_2^0 &= 3S_z^2 - X\\
O_2^2 &= \frac 1 2 (S_+^2+ S_-^2)\\
O_4^0 &= 35S_z^4 - (30X-25)S^2_z +(3X^2-6X)\\
O_4^4 &= \frac 1 2 (S_+^4+ S_-^4)
\end{align*}
where $X = S(S+1)I\in \mathbb R^{S(S+1)\times S(S+1)}$, and for $j = 1:21$ we have
\begin{align*}
    (S_z)_{j,j} \quad&= (S + 1-j)\\     (S_+)_{j,j+1} &=  \sqrt{j(2S+1-j)}\\
    (S_-)_{j+1,j} &=  \sqrt{j(2S+1-j)}.
\end{align*}
This system is overdetermined with all eigenvalues prescribed ($m=n=21$) and only $\ell=5$ parameters. The experimental eigenvalues are not given in \cite{bircher_transverse_2004} so we simulated them using the solution: $x^*_1=[-4594,-0.67,-0.7737,164.41]$ Hz using EasySpin (in fact there are four solutions described in the paper). The initial condition used was $x_0 = [-1000,1,1,1,0]$ Hz and all methods converged to the alternate solution $x^*_2 =[-4594,-0.67,1.2256,130.24]$ Hz. The summary of the computational cost of each method can be found in Table \ref{tab: Mn12} the time was averaged over $100$ iterations of each method from the same initial guess. Again the RGD LP methods are faster than LP, roughly twice as fast. In this case since all eigenvalues are prescribed $\rho$ is  known a priori, $\rho_i = i$, however RGD LP still calculates this from scratch which accounts for the extra $1.5\times10^{-5}$ seconds per iteration compared to RGD LP Min. 

 \begin{table}
     \centering
     \captionof{table}{Computational cost for Example \ref{Ex. Mn12}, with stopping tolerance $\epsilon = 1\times10^{-8}$ Hz}
     \label{tab: Mn12}
     \begin{tabular}{cccc}
          \toprule \mc{Algorithm} & \mc{No. Iterations} & \mc{CPU Time (seconds)}  & \mc{Time per Iteration}\\\hline
          RGD LP Min &  $ 136 $  & $ 0.011 $& $8.1\times 10^{-5} $ \\
          RGD LP          &   $136$   &  $0.013 $& $9.6\times 10^{-5}$   \\
          LP              &   $136$   &  $0.022 $& $1.6\times 10^{-4}$   \\\toprule
          \end{tabular}   
 \end{table}
\end{example} 
\begin{example}\label{Ex. Cr6}

The next example is a finite antiferromagnetic spin chain composed of six exchange coupled $S=\frac32$ Cr$3^+$ ions. The hamiltonian model for this example only requires two Steven's operators:
\begin{equation}
    A_1 = O^0_2, \, A_2 = O^2_2 \,  \in \mathbb R ^{4096\times4096}.
\end{equation}
However, because it contains more than one effective spin centre, each with a spin of $S = \frac32$,  the operators for the whole system are defined as Kronecker products of the operators for one spin centre. It is also necessary to include an electron-electron exchange operator that  includes the interaction terms between neighbouring spin centres (the contribution of all other interactions is negligible):
\begin{equation}
    A_3 =  \sum_{\text{neighbours}} S^i_xS^j_x + S^i_yS^j_y + S^i_zS^j_z \,  \in \mathbb R ^{4096\times4096}
\end{equation}
where $S_z^i$ is the $S_z$ spin operator for the $i$th spin centre,  defined by 
\begin{equation}
    S_z^i = I_{4}^{(i-1)\otimes} \otimes S_z \otimes I_{4}^{(N-i)\otimes}
\end{equation}
and similarly for $S_x$ and $S_y$ where the local operators are defined as 
\begin{align*}
    (S_x)_{j,j+1} &=(S_x)_{j+1,i}   =\frac \hbar 2 \sqrt{j(2S+1-j)}\\
    -(S_y)_{j,j+1} &= (S_y )_{j+1,i} = \frac {i\hbar} 2 \sqrt{j(2S+1-j)}.
\end{align*}
All of these operators are highly sparse matrices, as can be seen from their spyplots in Figure \ref{fig: SpyPlots}.
\begin{figure}
    \centering
    \includegraphics[width=0.8\linewidth]{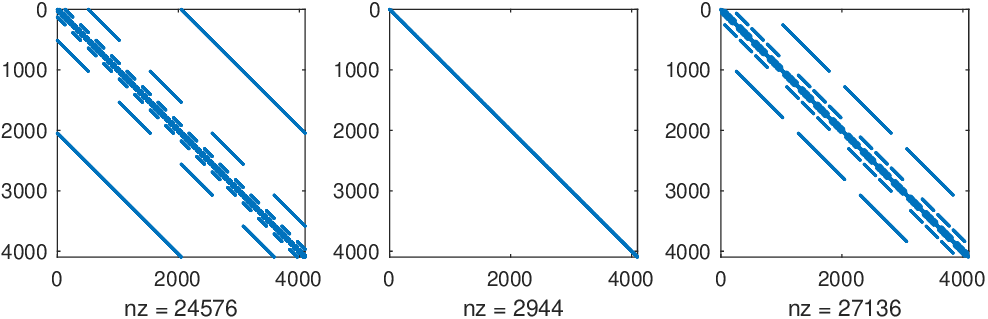}
    \caption{Spyplots showing the sparsity of $A_1,A_2$ and $A_3$ for Example \ref{Ex. Cr6}}
    \label{fig: SpyPlots}
\end{figure}
This is an overdetermined system with $m=21$ eigenvalues found experimentally and $ \ell=4$ parameters. The eigenvalues were simulated using EasySpin with the parameters given in \cite{baker_varying_2011} as the solution: $x^* = [1692.5,-3304.4,353000]$ Hz, and ground state $5211700$. To capture the behaviour of the methods across the whole solution space we calculated the iterations on a logarithmic grid from $1e3$ to $1e7$ (but with the second parameter negative) with a total of $625$ points. The timings for the calculation of all these steps is given in table \ref{tab: Cr6}.

   \begin{table}
     \centering
          \caption{Computational cost for Example \ref{Ex. Cr6}}
     \label{tab: Cr6}
     \begin{tabular}{c c c}
          \toprule \mc{Algorithm}  & \mc{CPU Time (seconds)}& \mc{Average Time per Iteration}  \\\hline
          RGD LP Min &  143   &  0.23 \\
          RGD LP          &  2710  & 4.34  \\
          LP              &  3730 &   5.97\\\toprule
          \end{tabular}

 \end{table}
By comparing this to Example \ref{Ex. toeplitz} it can be seen that the magnitude of the speedup of RGD Smallest compared to RGD LP is correlated with the size of the basis matrices. 
\end{example} 
 \begin{example}\label{Ex. Mn6}
The last INS example is that of Mn$_6$. This sample is modelled using the $O_2^0$ Steven's operator for the two MnIII ions and four electron-electron exchange operators between nearest neighbours. Due to the size of the basis matrices in this example, $A(x)\in\mathbb R^{32400\times32400}$, it is intractable to perform a full eigendecomposition so only the  RGD LP Min method is used.

There were 16 eigenvalues measured experimentally: $\lambda^* = [0,0.342,1.428,1.428,2.621,2.621,2.621,$ $3.417,3.417,5.6,5.6,5.6,5.6,5.6,6.097,6.097]\times10^5 \text{ Hz}$, it is therefore an overdetermined system with $m=16 > 5=\ell$. The initial guess in this case is given by $ x_0 = [100,100,100,100,100, 2\times10^7]$ Hz.
 \begin{table}
     \centering
          \caption{Computational cost for Example \ref{Ex. Mn6}.}
     \label{tab: Mn6}
     \begin{tabular}{cccc}
          \toprule Algorithm & No. Iterations & CPU Time (seconds) & \mc{Average Time per Iteration}  \\\hline
          RGD LP Min &       600    &   238 & 0.4\\\toprule
          \end{tabular}

 \end{table}


The linear convergence of the method for this example can be seen in Figure \ref{fig: Mn6 Convergence}.

\begin{figure}
\centering
\includegraphics[width=0.5\linewidth]{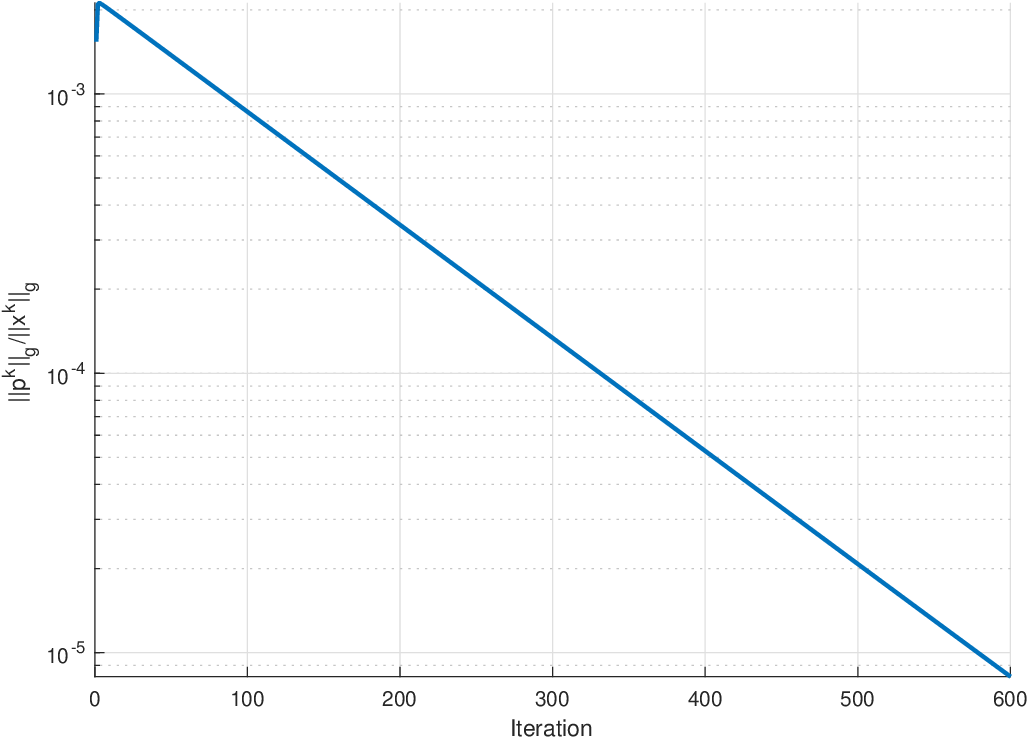}
    \caption{Convergence plot of the RGD LP Min method for Example  \ref{Ex. Mn6}}
    \label{fig: Mn6 Convergence}
\end{figure}
\end{example} 
\section{Conclusion}
In Sections \ref{sec. LP} and \ref{sec. Riemannian GD} we showed that the Lift and Projection method, described by Chen and Chu in \cite{chen_least_1996}, is equivalent to the Riemannian gradient descent method applied to the manifold $\mathbb R^\ell$  equipped with the metric induced by the map $A(x)$. Note that this map is what describes the structural constraints of the inverse problem, compared to the spectral constraints described by $\Lambda^*$.

The new interpretation of the method also allows for a more general function to minimised, for example it is possible to include additional least squares constraints. Each iterate of the new formulation is also quicker to compute,  especially in the case when only a small subset of eigenvalues is prescribed that are all computable as part of a partial eigendecomposition -- such as the case when only the smallest/largest in real value/magnitude eigenvalues are required. 

In Section \ref{sec. Global Convergence} an important relationship was established between the matrix $B$ from \eqref{eqn. Def B} and the  Hessian of $F$. It was then proved that the Lift and Projection method is globally convergent, a stronger result  than that proved in \cite{chen_least_1996}. Section \ref{sec. Numerical Experiments} contained several numerical examples that even when $\rho$ has to be calculated at each iteration  showed the RGD LP is a faster implementation of the Lift and Projection method. It was also shown that when the number of prescribed eigenvalues is small, and the permutation  $\rho$ is know  the RGD LP method is much faster; the cost scales roughly linearly with $n\times \mathrm{nnz}$, where $\mathrm{nnz}$ is the number of nozero entries of $A(x)$. While this speedup is not possible in the general case, many real world problem are of this form -- for example the INS examples discussed in this paper. 

\section*{Acknowledgements}
\sloppy ABR thanks the University of Manchester for a Dean’s Doctoral Scholarship. MW thanks the Polish National Science Centre (SONATA-BIS-9), project no. 2019/34/E/ST1/00390, for the funding that supported some of this research. 
\bibliographystyle{siamplain}
\bibliography{references}

\begin{thebibliography}{10}

\bibitem{baker_spectroscopy_2014}
{\sc M.~Baker}, {\em Spectroscopy methods for molecular nanomagnets}, in Structure and Bonding, Springer Berlin Heidelberg, 2014, pp.~231--291, \url{https://doi.org/10.1007/430_2014_155}, \url{https://link.springer.com/10.1007/430_2014_155}.

\bibitem{baker_neutron_2012}
{\sc M.~Baker and H.~Mutka}, {\em Neutron spectroscopy of molecular nanomagnets}, The European Physical Journal Special Topics, 213 (2012), pp.~53--68, \url{https://doi.org/10.1140/epjst/e2012-01663-6}, \url{http://link.springer.com/10.1140/epjst/e2012-01663-6}.

\bibitem{baker_varying_2011}
{\sc M.~L. Baker, A.~Bianchi, S.~Carretta, D.~Collison, R.~J. Docherty, E.~J.~L. McInnes, A.~McRobbie, C.~A. Muryn, H.~Mutka, S.~Piligkos, M.~Rancan, P.~Santini, G.~A. Timco, P.~L.~W. Tregenna-Piggott, F.~Tuna, H.~U. G{\"u}del, and R.~E.~P. Winpenny}, {\em Varying spin state composition by the choice of capping ligand in a family of molecular chains: detailed analysis of magnetic properties of chromium(iii) horseshoes}, Dalton Transactions, 40 (2011), p.~2725, \url{https://doi.org/10.1039/c0dt01243b}, \url{http://xlink.rsc.org/?DOI=c0dt01243b}.

\bibitem{bircher_transverse_2004}
{\sc R.~Bircher, G.~Chaboussant, A.~Sieber, H.~U. G{\"u}del, and H.~Mutka}, {\em Transverse magnetic anisotropy in mn 12 acetate: Direct determination by inelastic neutron scattering}, Physical Review B, 70 (2004), p.~212413, \url{https://doi.org/10.1103/PhysRevB.70.212413}, \url{https://link.aps.org/doi/10.1103/PhysRevB.70.212413}.

\bibitem{bloor_riley_deflation_2025}
{\sc A.~Bloor~Riley, M.~Webb, and M.~L. Baker}, {\em Deflation {Techniques} for {Finding} {Multiple} {Local} {Minima} of a {Nonlinear} {Least} {Squares} {Problem}}, Feb. 2025, \url{https://doi.org/10.48550/arXiv.2409.14438}, \url{http://arxiv.org/abs/2409.14438} (accessed 2025-04-07).
\newblock arXiv:2409.14438 [math].

\bibitem{boumal_introduction_2023}
{\sc N.~Boumal}, {\em An introduction to optimization on smooth manifolds}, Cambridge University Press, Cambridge ; New York, NY, 2023.

\bibitem{brockett_dynamical_1991}
{\sc R.~Brockett}, {\em Dynamical systems that sort lists, diagonalize matrices, and solve linear programming problems}, Linear Algebra and its Applications, 146 (1991), pp.~79--91, \url{https://doi.org/10.1016/0024-3795(91)90021-N}, \url{https://linkinghub.elsevier.com/retrieve/pii/002437959190021N}.

\bibitem{chen_least_1996}
{\sc X.~Chen and M.~T. Chu}, {\em On the least squares solution of inverse eigenvalue problems}, SIAM Journal on Numerical Analysis, 33 (1996), pp.~2417--2430, \url{https://doi.org/10.1137/S0036142994264742}, \url{http://epubs.siam.org/doi/10.1137/S0036142994264742}.

\bibitem{chu_inverse_2005}
{\sc M.~T.-C. Chu and G.~H. Golub}, {\em Inverse Eigenvalue problems: theory, algorithms, and applications}, Numerical mathematics and scientific computation, Oxford University Press, Oxford ; New York, 2005.
\newblock OCLC: ocm59877517.

\bibitem{dennis_numerical_1996}
{\sc J.~E. Dennis and R.~B. Schnabel}, {\em Numerical methods for unconstrained optimization and nonlinear equations}, Classics in applied mathematics, Society for Industrial and Applied Mathematics, Philadelphia, 1996.

\bibitem{deuflhard_newton_2011}
{\sc P.~Deuflhard}, {\em Newton Methods for Nonlinear Problems: Affine Invariance and Adaptive Algorithms}, vol.~35 of Springer Series in Computational Mathematics, Springer Berlin Heidelberg, 2011, \url{https://doi.org/10.1007/978-3-642-23899-4}, \url{https://link.springer.com/10.1007/978-3-642-23899-4}.

\bibitem{friedman_macroscopic_1996}
{\sc J.~R. Friedman, M.~P. Sarachik, J.~Tejada, and R.~Ziolo}, {\em Macroscopic measurement of resonant magnetization tunneling in high-spin molecules}, Physical Review Letters, 76 (1996), pp.~3830--3833, \url{https://doi.org/10.1103/physrevlett.76.3830}, \url{https://link.aps.org/doi/10.1103/PhysRevLett.76.3830}.

\bibitem{furrer_magnetic_2013}
{\sc A.~Furrer and O.~Waldmann}, {\em Magnetic cluster excitations}, Reviews of Modern Physics, 85 (2013), pp.~367--420, \url{https://doi.org/10.1103/RevModPhys.85.367}, \url{https://link.aps.org/doi/10.1103/RevModPhys.85.367}.

\bibitem{gatteschi_molecular_2006}
{\sc D.~Gatteschi, R.~Sessoli, and J.~Villain}, {\em Molecular nanomagnets}, Mesoscopic physics and nanotechnology, Oxford University Press, Oxford ; New York, 2006.
\newblock OCLC: ocm62133760.

\bibitem{higham_accuracy_2002}
{\sc N.~J. Higham}, {\em Accuracy and stability of numerical algorithms}, Other titles in applied mathematics, Society for Industrial and Applied Mathematics (SIAM), Philadelphia, Pa, 2nd ed.~ed., 2002, \url{https://doi.org/10.1137/1.9780898718027}.

\bibitem{lee_introduction_2018}
{\sc J.~M. Lee}, {\em Introduction to Riemannian manifolds}, vol.~176 of Graduate texts in mathematics, Springer, second edition~ed., 2018, \url{https://doi.org/10.1007/978-3-319-91755-9}.

\bibitem{nocedal_numerical_2006}
{\sc J.~Nocedal and S.~J. Wright}, {\em Numerical optimization}, Springer series in operations research, Springer, New York, 2nd ed.~ed., 2006.
\newblock OCLC: ocm68629100.

\bibitem{polizzi_density-matrix-based_2009}
{\sc E.~Polizzi}, {\em Density-matrix-based algorithm for solving eigenvalue problems}, Physical Review B, 79 (2009), p.~115112, \url{https://doi.org/10.1103/PhysRevB.79.115112}, \url{https://link.aps.org/doi/10.1103/PhysRevB.79.115112}.

\bibitem{rudowicz_generalization_2004}
{\sc C.~Rudowicz and C.~Y. Chung}, {\em The generalization of the extended stevens operators to higher ranks and spins, and a systematic review of the tables of the tensor operators and their matrix elements}, Journal of Physics: Condensed Matter, 16 (2004), pp.~5825--5847, \url{https://doi.org/10.1088/0953-8984/16/32/018}, \url{https://iopscience.iop.org/article/10.1088/0953-8984/16/32/018}.

\bibitem{sessoli_magnetic_1993}
{\sc R.~Sessoli, D.~Gatteschi, A.~Caneschi, and M.~A. Novak}, {\em Magnetic bistability in a metal-ion cluster}, Nature, 365 (1993), pp.~141--143, \url{https://doi.org/10.1038/365141a0}, \url{https://www.nature.com/articles/365141a0}.

\bibitem{sleijpen_jacobi--davidson_2000}
{\sc G.~L.~G. Sleijpen and H.~A. Van~der Vorst}, {\em A {Jacobi}--{Davidson} {Iteration} {Method} for {Linear} {Eigenvalue} {Problems}}, SIAM Review, 42 (2000), pp.~267--293, \url{https://doi.org/10.1137/S0036144599363084}, \url{http://epubs.siam.org/doi/10.1137/S0036144599363084}.

\bibitem{stewart_krylov--schur_2002}
{\sc G.~W. Stewart}, {\em A {Krylov}--{Schur} {Algorithm} for {Large} {Eigenproblems}}, SIAM Journal on Matrix Analysis and Applications, 23 (2002), pp.~601--614, \url{https://doi.org/10.1137/S0895479800371529}, \url{http://epubs.siam.org/doi/10.1137/S0895479800371529} (accessed 2025-04-03).

\bibitem{stoll_easyspin_2006}
{\sc S.~Stoll and A.~Schweiger}, {\em Easyspin, a comprehensive software package for spectral simulation and analysis in epr}, Journal of Magnetic Resonance, 178 (2006), pp.~42--55, \url{https://doi.org/10.1016/j.jmr.2005.08.013}, \url{https://linkinghub.elsevier.com/retrieve/pii/S1090780705002892}.

\bibitem{zhao_riemannian_2022-1}
{\sc Z.~Zhao, X.-Q. Jin, and T.-T. Yao}, {\em The {Riemannian} two-step perturbed {Gauss}–{Newton} method for least squares inverse eigenvalue problems}, Journal of Computational and Applied Mathematics, 405 (2022), p.~113971, \url{https://doi.org/10.1016/j.cam.2021.113971}, \url{https://linkinghub.elsevier.com/retrieve/pii/S0377042721005707} (accessed 2025-04-09).

\bibitem{zhao_riemannian_2022}
{\sc Z.~Zhao, X.-Q. Jin, and T.-T. Yao}, {\em A {Riemannian} under-determined {BFGS} method for least squares inverse eigenvalue problems}, BIT Numerical Mathematics, 62 (2022), pp.~311--337, \url{https://doi.org/10.1007/s10543-021-00874-z}, \url{https://link.springer.com/10.1007/s10543-021-00874-z} (accessed 2025-04-09).

\end{thebibliography}
\appendix
\end{document}